\documentclass[reqno,10pt]{amsart}
\usepackage{amsmath,amsthm,amsfonts,color,graphicx}
\usepackage[latin1]{inputenc}
\usepackage[makeroom]{cancel}
\usepackage{tikz,bbold}
\usetikzlibrary{decorations.pathreplacing}

\newcommand{\bbn}{\mathbb{N}}

\newcommand{\dhr}{\mathrel{\lhook\joinrel\relbar\kern-.8ex\joinrel\lhook\joinrel\rightarrow}}

\newcommand{\seqk}[1]{(#1 _k)_{k\in\bbn}}

\newtheorem{thm}{Theorem}[section]
\newtheorem{lem}[thm]{Lemma}
\newtheorem{cor}[thm]{Corollary}
\newtheorem{prop}[thm]{Proposition}

\newtheorem{rem}[thm]{Remark}

\DeclareMathAlphabet{\mathpzc}{OT1}{pzc}{m}{it}

\begin{document}
\bibliographystyle{plain}

\title[]{Domain Variations and Moving Boundary Problems}

\author{Patrick Guidotti}
\address{University of California, Irvine\\
Department of Mathematics\\
340 Rowland Hall\\
Irvine, CA 92697-3875\\ USA }
\email{gpatrick@math.uci.edu}

\begin{abstract}
In the past few decades maximal regularity theory has successfully
been applied to moving boundary problems. The basic idea is 
to reduce the system with varying domains to one in a fixed
domain. This is done by a transformation, the so-called Hanzawa
transformation, and yields a typically nonlocal and nonlinear coupled
system of (evolution) equations. Well-posedness results can
then often be established as soon as it is proved that the
relevant linearization is the generator of an analytic semigroup or
admits maximal regularity. To implement this program, it is necessary
to somehow parametrize to space of boundaries/domains (typically the
space of compact hypersurfaces $\Gamma$ in $\mathbb{R}^n$, in the
Euclidean setting). This has traditionally been achieved by means of
the already mentioned Hanzawa transformation. The approach, while
successful, requires the introduction of a smooth manifold
$\Gamma_\infty$ close to the manifold $\Gamma_0$ in which one cares to
linearize. This prevents one to use coordinates in which $\Gamma_0$
lies at their ``center''. As a result formul{\ae} tend to contain
terms that would otherwise not be present were one able to linearize
in a neighborhood emanating from $\Gamma_0$ instead of from
$\Gamma_\infty$. In this paper it is made use of flows (curves of
diffeomorphisms) to obtain a general form of the relevant
linearization in combination with an alternative coordinatization of
the manifold of hypersurfaces, which circumvents the need for the
introduction of a ``phantom'' reference manifold $\Gamma_\infty$ by,
in its place, making use of a ``phantom geometry'' on $\Gamma_0$. The
upshot is a clear insight into the structure of the linearization,
simplified calculations, and simpler formul{\ae} for the resulting
linear operators, which are useful in applications.
\end{abstract}

\keywords{}
\subjclass[1991]{}

\maketitle

\section{Introduction}
Moving boundary problems are ubiquitous and numerous in
applications. Latter include, but are by no means limited to, fluid 
dynamics with, e.g., the classical Stefan problem, and biology with,
e.g models of tumor growth.

In abstract terms such problems consist of a system of (initial)
boundary value problems for unknown physical quantities (read
concentrations, velocity fields, temperature, ...) and for at least
one unknown (evolving) domain in which the boundary value problems are
set. Even when the equations for the unknown physical quantities
appear linear, the system is not, due to the coupling with the
geometry. Indeed, two solutions living on two distinct domains can not
be added to obtain a new solution on a new domain.

A versatile general purpose approach to (fully) nonlinear evolution
equations of parabolic type is given by {\em  optimal (also maximal)
  regularity theory}, see e.g. \cite{Lun95,Ama05,Ama05b}. In a
nutshell, the approach consists in
linearizing a nonlinear equation/system in a point in the space of
unknowns, prove that the linearization is an isomorphism (between
carefully chosen function spaces), and eventually solving the
equations by a perturbation argument.

In the context of free and moving boundary problems, linearization in
the unknown necessarily includes taking domain variations (recall that
the domain is itself an unknown of the problem). This amounts to
measuring the infinitesimal dependence of functions, differential
operators, pseudo-differential operators, and geometric quantities on
the domains on which they are defined.

To be more specific consider a domain $\Omega_0$ in $\mathbb{R}^n$,
$n\in\mathbb{N}$, defined by its boundary $\Gamma_0$ as the bounded
region inside of it.  It is supposed that $\Gamma_0$ be a compact
hypersurface of limited regularity, say $\operatorname{C}^2$, for
now. For technical reasons that will become more explicit shortly, the
parametrization problem has traditionally been solved by introducing
coordinates in a neighborhood of $\Gamma_0$ (it is clearly enough to
vary the boundary as a means to vary the domain) based on a smooth
($\operatorname{C}^\infty$ or analytic) manifold $\Gamma_\infty$
arbitrarily close to $\Gamma_0$ (in the $\operatorname{C}^2$
sense). The basic idea consists in parametrizing the surface
$\Gamma_0$ over $\Gamma_\infty$ as a graph in ``normal 
direction'', that is, by a function 
$$
 \rho_0:\Gamma_0\to\mathbb{R}
$$
via
$$
 \Gamma_0=\big\{ y+\rho_0(y)\nu _\infty(y)\,\big |\, y\in
 \Gamma_\infty\big\},
$$
where $\nu_\infty$ is the (smooth) outward unit normal to
$\Gamma_\infty$. In this way, the unknown domain can be described
(locally in time, but this is enough) as an unknown function and the
geometry (read $\nu_\infty$) does not impose any limitations since it
is taken to be smooth. Notice that it would be impossible to choose
$\Gamma_0$ as a reference manifold since it would require one to use
its unit outward normal field $\nu_0$, which enjoys one
less degree of regularity as compared to the manifold itself. As shall become
evident later, this loss of regularity cannot be afforded if one is to
take the optimal regularity approach briefly sketched above. This is
the core idea of the transformation, which was first employed for the Stefan
problem \cite{Hanz81}, and that has become known as the Hanzawa
transformation. This approach has been used repeatedly and was nicely
expounded in \cite{PS13}.

It is the purpose of this paper to overcome the ``regularity issue''
in an alternative way that does not require the use of a smooth
reference manifold, but rather uses the surface $\Gamma_0$ as the
center of the coordinate patch, in which, after all, the linearization is
needed. The idea can be simply stated: instead of using a smooth
``phantom manifold'' $\Gamma_\infty$, introduce a regularized normal
field $\nu ^\delta_0$ on $\Gamma_0$ and use it to parametrize a
neighborhood of $\Gamma_0$ in the space of surfaces. This can be
thought of as using a ``phantom geometry'' on $\Gamma_0$. Notice that
in the smooth case, the two approaches coincide, since $\Gamma_0$ with
its natural geometry can always be chosen as reference manifold. In
this respect, the two approaches coincide in the smooth context. 
An additional goal of this paper is to offer a more geometrical
approach to the issue of linearization. It uses flows and, more in
general, curves of diffeomorphisms to conveniently identify it. This
results in simpler and more transparent  
calculations which can be performed before the parametrization
described above for the unkown surface is introduced, at the very end,
in order to obtain a system of PDEs for unknown functions only.

An added advantage of the approach is the simplified form taken by the
linearization which significantly shortens the analysis required to prove
that it is a generator of an analytic semigroup, enjoys maximal
regularity, or to obtain spectral information for stability
analysis. Prototypical examples are discussed at the end of the paper. 
\section{Preliminaries}
Basic facts from differential geometry and manifold theory will be
used freely in the sequel. It is referred to standard references such
as \cite{Lee97,Lee02,Tu11} for the 
required background. For $\alpha\in(0,1)$ and $k\in \mathbb{N}$, 
denote by $\mathcal{M}^{k+\alpha}$ the space of embedded hypersurfaces
of $\mathbb{R}^n$ given by 
\begin{equation}\label{hyper}
  \mathcal{M}^{k+\alpha}=\big\{ \Gamma \subset \mathbb{R}^n\,\big |\,
  \Gamma\text{ compact, orientable hypersurface of class }
  buc^{k+\alpha}\bigr\},
\end{equation}
for $k\in \mathbb{N}$ and where the regularity space $buc^{k+\alpha}$
is the so-called little H\"older space. For an open subset $O\subset
\mathbb{R}^n$, the latter is defined as the closure of the regular
space of bounded and uniformly H\"older continuous functions given by
$$
 \operatorname{BUC}^{k+\beta }(O)=\big\{ f:O\to \mathbb{R}\, \big |\, 
 f\in \operatorname{BUC}^k(O)\text{ and }\partial^\gamma f\in
 \operatorname{BUC}^\beta(O)\text{ for }|\gamma|=k\big\},
$$
with $\beta>\alpha$, in the topology determined by the norm $\|\cdot\|
_{2+\alpha,\infty}$ defined through
$$
 \| f\| _{k+\alpha,\infty}=\max_{|\gamma|\leq k}\| \partial ^\gamma
 f\| _\infty+\max_{|\gamma|=k}[\partial^\gamma f]_\alpha,
$$
where
$$
 [g]_\alpha=\sup _{x\neq y}\frac{|g(x)-g(y)|}{|x-y|^\alpha },\: g\in
 \operatorname{C}(O). 
$$
If $O$ is replaced by a compact manifold $M\in buc^{k+\alpha}$, then
the spaces $buc^{l+\beta}(M)$, for $l\leq k$ and $\beta\in(0,1)$ with
$\beta\leq \alpha$ if $l=k$, are defined in the standard way by
resorting to localizations combined with a smooth partition of unity.
\begin{rem}
The choice of little H\"older spaces is motivated by the ease provided by
the use of a family of function spaces with dense embeddings in one
another in the context of maximal regularity for generators of
analytic semigroups. The space $buc^{\alpha}(M)$ consists of those
$\operatorname{BUC}^\alpha(M)$ functions $g$ satisfying
$$
 \lim _{\delta\to 0}\sup _{y\neq z\in \mathbb{B}_M(x,\delta)}
 \frac{|g(y)-g(z)|}{d_M(y,z)^\alpha}=0.
$$
For all considerations preceding the final examples, they can be
replaced by the more standard classes of H\"older regularity
$BUC^{k+\alpha}$, of which they are closed subspaces.
\end{rem}
While this choice of spaces is not essential until maximal
regularity results are used and, even then not unique, it is made for
consistency with the final part of the paper and for simplicity of
presentation. It is referred to \cite{Ama16} and the references cited
therein for alternative functional settings in which maximal
regularity holds. As the
preferred spaces are a matter of taste and not of necessity in most
applications, the choice made here is not restrictive but allows for a
more concise presentation. 

The Haussdorff distance on compact subsets defined by
$$
 d_{\operatorname{C}^0}(K,\overline{K})=\max\big\{
 \max_{\overline{x}\in\overline{K}}d(\overline{x},K),\max_{x\in
   K}d)x,\overline{K})\big\}
$$
can be used to define a distance $d_{\operatorname{C}^1}$ between
$\Gamma,\overline{\Gamma}\in 
\mathcal{M}^{1+\alpha}$ in the following manner
\begin{equation}\label{dist}
  d_{\operatorname{C}^1}(\Gamma,\overline{\Gamma})=
  d_{\operatorname{C}^0}(N\Gamma,N\overline{\Gamma}) 
\end{equation}
where
$$
 N \Gamma =\big\{ \bigl(y,\nu_\Gamma(y)\bigr)\,\big |\, y\in
 \Gamma\big\}\subset \mathbb{R}^{2n},
$$
where $\nu_\Gamma(y)$ denotes the unit, outward pointing normal to
$\Gamma$ at $y$. 
Proximity in $d_{\operatorname{C}^1}$ therefore implies not only that the hypersurfaces
are close to each other but that also their tangent spaces cross
everywhere at a uniformly small angle. This is used to exclude ``rough'' (oscillatory)
approximations. Given $\Gamma_0\in \mathcal{M}^{2+\alpha}$, a little
room is needed in which to operate. It is provided by the following
lemma. 
\begin{lem}[Existence of a tubular neighborhood]\label{tub}
Given $\Gamma_0\in \mathcal{M}^{2+\alpha}$, there is $r_0>0$ such that
$$
 T_{r_0}(\Gamma_0):=\{ x\in \mathbb{R}^n\, |\, d(x,\Gamma_0)<r_0\}
$$
is an open neighborhood of $\Gamma_0$ diffeomorphic to
$\Gamma_0\times(-r_0,r_0)$. 
\end{lem}
Notice that $d(\cdot,\Gamma_0)$ will always denote the signed
distance to $\Gamma_0$ with the understanding that it is negative
in the interior of the domain bounded by $\Gamma_0$.
\begin{proof}
While the proof is well-known, it is given anyway as a way to introduce
some notation which will be useful again later.

By assumption $\Gamma_0$ has bounded principal curvatures. Fix a point
$y\in \Gamma_0$ and choose coordinates $s=(s^1,\dots,s^{n-1})$ such
that $\tau_j^0=\frac{\partial }{\partial s^j},\: j=1,\dots,n-1$, is a
orthonormal basis of $T_y \Gamma_0$ 
consisting of principal directions, i.e. satisfying
$$
 d_{\tau ^0_j}\nu _{0} =\frac{d}{ds^j}\big |_{s=0}\nu _{0}=\lambda
 _j^0\tau ^0_j\text{ for } j=1,\dots,n-1,
$$
where $\nu_0=\nu_{\Gamma_0}$ and  the dependence on $y$ or $s$ is
omitted, and $\lambda _j^0$ are the principal curvatures of $\Gamma_0$
at $y$. It is assumed that $\tau_1^0,\dots,\tau_{n-1}^0,\nu_0$ is a positively
oriented orthonormal basis of $T_y \mathbb{R}^n$. Define the map
$$
 \Phi:\Gamma_0\times (-r_0,r_0)\to \mathbb{R}^n,\: (y,r)\mapsto
 y+r\nu_0(y),
$$
and notice that $\Phi\in buc^{1+\alpha}\bigl(\Gamma_0\times(-r_0,r_0)
\bigr)$. It follows from the choice of coordinates that
\begin{align*}
  \frac{\partial\Phi}{\partial s^j}&=\tau ^0_j+rd_{\tau^0_j}\nu=(1+r\lambda
  _j^0)\tau_j^0,\\
  \frac{\partial\Phi}{\partial r}&=\nu_0.
\end{align*}
Then one has that
\begin{align*}
  \frac{\partial\Phi}{\partial s^j}\cdot \frac{\partial\Phi}{\partial
  s^k}&=\delta_{jk}(1+r\lambda _j^0)^2,\: j,k=1,\dots, n-1,\\
  \frac{\partial\Phi}{\partial s^j}\cdot \frac{\partial\Phi}{\partial
  r}&=0,\: j=1,\dots,n-1,\\
  \frac{\partial\Phi}{\partial r}\cdot \frac{\partial\Phi}{\partial r}&=1.
\end{align*}
By assumption 
$$
 \max_{j=1,\dots, n-1}|\lambda _j^0|\leq \Lambda<\infty
 \text{ on }\Gamma_0,
$$
and, consequently, $D\Phi(y,r)$ is invertible for $0\leq
r<\tilde r_0$ and some $\tilde r_0>0$ which is taken to coincide with
$r_0$ witout loss of generality. This holds independently of the
point $y\in \Gamma_0$. Compactness and the inverse function theorem
then imply that
$$
 \Phi\big | _{\mathbb{B}_{\Gamma_0}(y_j,r_0)\times(-r_0,r_0)}
$$
is a diffeomorphism onto its image and
$$
 \cup _{l=1,\dots,N}\mathbb{B}_{\Gamma_0}(y_l,r_0)\supset \Gamma_0,
$$
for some $N\in \mathbb{N}$ and $y_l\in \Gamma_0$, $l=1,\dots,N$. It
remains to make sure that hypersurface does not come close to itself
(not in a local fashion, but rather in a global way) in order to
obtain a global diffeomorphism. To that end, define
$$
 \sigma_l=\inf _{y\in
   \mathbb{B}_{\Gamma_0}(y_l,r_0)^\mathsf{c}}d_{\mathbb{R}^n}(y,y_l) 
$$
and reset $r_0$ to half of $\sigma=\min_{l=1,\dots,N}\sigma_l$. Then
$\Phi\big |_{\Gamma_0\times(-r_0,r_0)}$ is injective as
desired. Indeed, if
$$
 \Phi(y_1,r_1)=\Phi(y_2,r_2)=x \text{ for } (y_i,r_i)\in
 \Gamma_0\times(-r_0,r_0),\: i=1,2,
$$
then 
$$
 d_{\mathbb{R}^n}(y_1,y_2)\leq
 d_{\mathbb{R}^n}(y_1,x)+d_{\mathbb{R}^n}(x,y_2)<\sigma, 
$$
so that $y_1,y_2$ must be in the same ball and thus coincide along
with $r_1=r_2$.
\end{proof}
\begin{rem}
Observe that the above construction yields a foliation of the tubular
neighborhood by $buc^{1+\alpha}$ surfaces only, since it employs the
normal of $\Gamma_0$.
\end{rem}
\begin{rem}\label{distsmooth}
The map $\Phi$ defined in the above proof yields coordinates $(y,r)$
in $T_{r_0}(\Gamma_0)$. In these variables it holds that
$d((r,y),\Gamma_0)=r$ for the signed distance function to
$\Gamma_0$. It readily follows that
$$
 \nabla d(\cdot,\Gamma_0)=1\, \frac{\partial }{\partial r}=\nu_0\in
 buc^{1+\alpha}\bigl( T_{r_0}(\Gamma_0)\bigr) , 
$$
This shows that $d(\cdot,\Gamma_0)\in
buc^{2+\alpha}\bigl( T_{r_0}(\Gamma_0)\bigr) $. Morever
$$
 \Delta d(\cdot,\Gamma_0)=H_{\Gamma_0},
$$
where $H$ is the mean curvature of $\Gamma_0$.
\end{rem}
The next lemma gives a refined version of the above which preserves
regularity. 
\begin{lem}\label{smoothtub}
Given $\Gamma_0\in \mathcal{M}^{2+\alpha}$, there is $r_0>0$ and
hypersurfaces $\Gamma_r\in \mathcal{M}^{2+\alpha}$ for
$r\in(-r_0,r_0)$ such that
$$
 \bigcup _{|r|<r_0}\Gamma_r
$$
is an open neighborhood of $\Gamma_0$.
\end{lem}
\begin{proof}
By Lemma \ref{tub} there is $\tilde r_0>0$ such that, given any $ x\in
T_{\tilde r_0}(\Gamma_0)$, there is one
$$
 (y,r)=\bigl( y(x),r(x)\bigr)\text{ s.t. }x=y+r\nu _{\Gamma_0}(y). 
$$
In $T_{\tilde r_0}(\Gamma_0)$ define the field
$$
 \widetilde{\widetilde{\nu}}(x)=\nu_{\Gamma_0}\bigl( y(x)\bigr),
$$
take a smooth cut-off function $\eta:\mathbb{R}\to \mathbb{R}$ with
$$
 0\leq\eta\leq 1,\: \eta |_{[-\tilde{r}_0/2,\tilde{r}_0/2]}\equiv 1,\text{
   and }\eta |_{(-3\tilde{r}_0/4,3\tilde{r}_0/4)^\mathsf{c}}\equiv 0,
$$
and set
$$
 \widetilde{\nu}(x)=
 \begin{cases}
   \widetilde{\widetilde{\nu}}(x)\eta \bigl( r(x)\bigr) ,&x\in
   T_{\tilde r_0}(\Gamma_0),\\
  0,&x\notin T_{\tilde r_0}(\Gamma_0).
 \end{cases}
$$
Then $\widetilde{\nu}\in buc^{1+\alpha}(\mathbb{R}^n,\mathbb{R}^n)$ is
a global vector field. Now take a compactly supported smooth mollifier
$\psi_\delta$ and define
$$
 \nu^\delta =\psi_\delta *\widetilde{\nu}
$$
componentwise. It follows that $\nu^\delta \in
BUC^\infty(\mathbb{R}^n,\mathbb{R}^n)$, that
$\operatorname{supp}(\nu^\delta)\subset T_{\tilde r_0}(\Gamma_0)$, and
that  
$$
 \nu ^\delta \to \nu \bigl( y(\cdot)\bigr) \text{ in
 }buc^{1+\alpha}\bigl( T_{\tilde r_0/2}(\Gamma_0)\bigr).
$$
In particular
$$
 |\nu ^\delta (y)\cdot\tau_{\Gamma_0}(y)|\leq c(\delta),\:\forall
 \tau_{\Gamma_0}(y)\in T_y \Gamma _0 \text{ with }|\tau_{\Gamma
   _0}(y)|=1,
$$
where $c(\delta)\to 0$, as $\delta$ tends to zero, uniformly in $y\in
\Gamma _0$. The vector field is therefore uniformly transversal to
$\Gamma_0$. Finally set
$$
 \Gamma_r=\varphi ^\delta(\Gamma_0,r)
$$
for the flow generated by the ode
$$
\begin{cases}
  \dot x=\nu^\delta(x),&\\
  x(0)=y\in \Gamma_0.&
\end{cases}
$$
It is easily seen that there is $r_0>0$ such that
$$
 \varphi ^\delta(\Gamma_0,r)=\Gamma_r\subset T_{\tilde
   r_0}(\Gamma_0),\: |r|\leq r_0,
$$
if $\delta<<1$, and standard ODE arguments yield that
$$
 \varphi ^\delta:
 \Gamma_0\times(-r_0,r_0)\to\bigcup_{r\in(-r_0,r_0)}\Gamma_r  
$$
is a diffeomorphism.
\end{proof}
The previous lemma provides coordinates $(y,r)$ for a neighborhood of
$\Gamma_0$, which can be denoted by $T^{\nu^\delta }_{r_0}(\Gamma_0)$
since it is constructed starting with the smooth vector field
$\nu^\delta$. Explicitly this means that
$$
 \forall\, x\in T^{\nu^\delta }_{r_0}(\Gamma_0)\:\exists !\, (y,r)\text{
   s.t. }x=\varphi^\delta(y,r).
$$
\begin{lem}\label{rhoparam}
Let $\Gamma_0\in \mathcal{M}^{j+\alpha}$, $j\geq 2$, and $\Gamma\in
\mathcal{M}^{k+\alpha}$, $k\geq 1$, satisfy
$d_{\operatorname{C}^1}(\Gamma,\Gamma_0)<<1$. Then there is a unique $\rho\in
buc^{k+\alpha}(\Gamma_0)$ such that
$$
 \Gamma =\big\{ \varphi^\delta \bigl(y,\rho(y)\bigr) \,\big |\, y\in
 \Gamma_0\big\}. 
$$
In more suggestive terms, $\Gamma$ can be viewed as a
$\nu^\delta$-graph (or a $\varphi^\delta$-graph) over $\Gamma_0$.
\end{lem}
\begin{proof}
Without loss of generality assume that $j=2$ and let $k=2$
first. Given $\Gamma$ with the above properties, it 
immediately follows from 
Lemma \ref{smoothtub} that, given any $x\in \Gamma$, there is a unique
$\bigl( y(x),r(x)\bigr)\in \Gamma_0\times(-r_0,r_0)$ such that
$$
 x=\varphi^\delta \bigl( y(x),r(x)\bigr).
$$
The function $\rho(y)$ is then obtained by setting $\rho(y)=r(x)$ if
$y=y(x)$. To show that it is well-defined and it has the required
smoothness consider the map
$$
 d^\delta_\Gamma : \Gamma_0\times(-r_0,r_0)\to
 \mathbb{R},\: (y,r)\mapsto d\bigl(\varphi^\delta(y,r),\Gamma\bigr),
$$
Observing that
$$
 \frac{\partial}{\partial r} d^\delta_\Gamma (y,r)=\nabla d
 \bigl( \varphi ^\delta(y,r),\Gamma)\cdot\dot\varphi^\delta (r,y),
$$
where the dot means differentiation with respect to $r$, the
assumption implies that 
$$
 \nabla d\bigl( \varphi ^\delta(y,r),\Gamma)=\nu _\Gamma(x)\simeq
 \nu_{0}(y) 
$$
if $\varphi^\delta(r,y)=x\in \Gamma$, and thanks to $\nabla d(\cdot,\Gamma)=\nu
_\Gamma$ on $\Gamma$ (see Remark \ref{distsmooth}). Simultaneously
$$
 \dot\varphi^\delta (r,y)=\nu^\delta\bigl(\varphi^\delta (r,y)\bigr)\simeq \nu_{0}(y).
$$
It follows that $ \frac{\partial}{\partial r} d^\delta_\Gamma
(y,r)\neq 0$ where $\varphi^\delta(r,y)\in \Gamma$. The implicit
function theorem now implies the existence of
$$
 \rho\in buc^{k+\alpha}(\Gamma_0)\text{ s.t. }d_\Gamma \bigl(
 \varphi^\delta(y,\rho(y))\bigr)\equiv 0,
$$
as claimed. The graph representation is obtained locally at first. If
it is extended to its maximal domain of validity, one can easily see
that the latter is open and closed  and, hence, coincides with the whole hypersurface
$\Gamma_0$ thanks to the fact that it is connected. Assume next that
$k=1$. Approximate $\Gamma$ in $buc^{1+\alpha}$ by a
family of $buc^{2+\alpha}$ hypersurfaces
$(\Gamma^\eta)_{\eta\in(0,1]}$. This can be done, as for instance in
\cite{BEL12}, by solving
$$
 \begin{cases}\Delta u=1&\text{in }\Omega,\\
 u=0&\text{on }\Gamma=\partial \Omega,
\end{cases} 
$$
and setting $\Gamma^\eta=[u=\eta]$. For each $\eta\in(0,1]$, by the first
part of the proof, there is a function $\rho^\eta\in
buc^{2+\alpha}(\Gamma_0)$ with the property that
$$
 \Gamma^\eta=\bigl(\varphi^\delta\circ(\operatorname{id},\rho^\eta)\bigr)
 (\Gamma_0).
$$
Now $d_{\operatorname{C}^1}(\Gamma,\Gamma_0)<<1$, the fact that
$\|\rho^\eta\|_\infty<\infty$ uniformly in $\eta$, and the
$\operatorname{C}^1$ 
convergence of $\Gamma^\eta$ to $\Gamma$ implies that necessarily
$$
 \| \rho^\eta\| _\infty+\| d\rho^\eta\|_\infty\leq c<\infty\text{ for }\eta\in(0,1],
$$
for some positive constant $c$. If this were not the case, then at
least one of the tangent vectors
\begin{equation}\label{tangvecs}
 \widetilde{\tau}^\eta_j=d\bigl[\varphi^\delta \circ
 (\operatorname{id},\rho^\eta)\bigr](\tau^0_j)=(d\varphi^\delta)\circ
 (\operatorname{id},\rho^\eta)(\tau^0_j)+\bigl[\nu^\delta\circ
 (\operatorname{id},\rho^\eta)\bigr]d\rho ^\eta(\tau_j^0),
\end{equation}
where $\tau_j^0$, $j=1,\dots,n$, is a basis of $T \Gamma_0$, would
eventually point, somewhere, in direction  of $\nu^\delta\simeq
\nu_0$. This follows from the fact that the first summand in the
right-hand-side of \eqref{tangvecs} remains bounded by construction, while
$\partial_j\rho^\eta=d\rho^\eta(\tau^0_j)$ would, for at least one $j$, tend to infinity in
size along a sequence $\seqk{y}$ of points in $\Gamma_0$. As this
sequence can be taken to converge to a point on $\Gamma_0$ without
loss of generality, a contradiction would ensue to the assumption that
$d_{\operatorname{C}^1}(\Gamma,\Gamma_0)<<1$. The Arz\'ela-Ascoli Theorem then implies
the existence of a continuous limiting function
$\rho_\Gamma:\Gamma_0\to \mathbb{R}$ such that $\rho^{\eta_k}\to\rho
_\Gamma$, as $k\to\infty$, for a sequence of indices $(\eta_k)_{k\in
  \mathbb{N}}$. It must then hold that
$$
 \Gamma=\bigl(\varphi^\delta\circ(\operatorname{id} ,\rho_\Gamma)\bigr) 
 (\Gamma_0) 
$$
and that $\rho\in buc^{1+\alpha}(\Gamma_0)$ due to the regularity of
  $\varphi^\delta$ and that of $\Gamma$ itself.
\end{proof}
The above lemma shows that, given $\rho\in buc^{2+\alpha}$ small
enough (in the $\operatorname{C}^1$ topology), the hypersurface
$$
 \Gamma_\rho=\big\{ \varphi^\delta \bigl(y,\rho(y)\bigr) \,\big |\, y\in
 \Gamma_0\big\}
$$
is well-defined.

It is important to have access to relevant geometric
quantities for $\Gamma_\rho$. Fix $y\in \Gamma_0$ and choose
again coordinates $s=(s^1,\dots,s^{n-1})$ along the principal
directions of $\Gamma_0$ at $y$ (just as in the proof of Lemma
\ref{tub} and using the notation introduced there). One computes that
\begin{equation}\label{taurho}
 \widetilde{\tau}^\rho_j=\frac{\partial}{\partial s^j}\varphi^\delta \circ
 (\operatorname{id},\rho)=\partial_j\varphi^\delta \circ
 (\operatorname{id},\rho)+\dot\varphi^\delta\circ
 (\operatorname{id},\rho)\partial_j\rho  
\end{equation}
is a tangent vector to $\Gamma_\rho$ at $\varphi^\delta \bigl(
y,\rho(y)\bigr)$. Observe that the notation
$$
 \partial_j g(y)=\langle d_yg,\tau_j^0 \rangle,\: y\in \Gamma _0
$$
was used in the above expressions for functions defined on
$\Gamma_0$. For $\delta<<1$ one has that 
$$
 \frac{\partial}{\partial s^j}\varphi^\delta \circ
 (\operatorname{id},\rho)\simeq (1+\lambda
 _j^0\rho)\tau^0_j+(\partial_j\rho)\,\nu _0,
$$
since
$$
\varphi^\delta(y,r)\simeq y+r\nu^\delta(y)\simeq y+r\nu_0(y).
$$
It can be concluded that
$$
 \widetilde{\tau}^\rho_1,\dots,\widetilde{\tau} ^\rho _{n-1}
$$
is a basis of $T_x{\Gamma_\rho}$ for $x=\varphi^\delta \bigl(
y,\rho(y)\bigr)$, provided that, as it is assumed, $\rho$ is small in
the $\operatorname{C}^1$ topology. 
\section{Taking variations by Flows}
Of interest is the dependence of various quantities on the
domain/manifold on which they are defined. Fix a compact oriented
hypersurface $\Gamma_0\in \mathcal{M}^{2+\alpha}$ and, for now, let
$F$ be any smooth section of a bundle over $\mathcal{M}^{2+\alpha}$,
which, in fact, can be assumed to be defined in a neighborhood
$\mathcal{U}^{2+\alpha}$ of $\Gamma_0$ only. In particular it will be
useful to have a convenient way to compute
$\left.\frac{d}{d \Gamma}\right|_{\Gamma=\Gamma_0}F$. A natural way to
do this is to fix a $\operatorname{C}^\infty$-flow $\varphi$ on
$\mathbb{R}^n$, that is, a smooth map
$$
 \varphi :(-\varepsilon,\varepsilon)\times \mathbb{R}^n\to
 \mathbb{R}^n, (s,x)\mapsto \varphi(s,x)=:\varphi _s(x),
$$
with $\varphi_s\in \operatorname{Diff}^\infty(\mathbb{R}^n)$ and
satisfying
$$
\begin{cases}
  \varphi _0=\operatorname{id}_{\mathbb{R}^n},&\\
  \varphi _{s+\tilde s}=\varphi_s\circ\varphi _{\tilde s},\: s,\tilde
  s,s+\tilde s\in (-\varepsilon,\varepsilon).
\end{cases}
$$
and use it in order to generate a curve of hypersurfaces in
$\mathcal{U}^{2+\alpha}$ by setting
$$
 \Gamma_s=\varphi_s(\Gamma_0),\: s\in (-\varepsilon,\varepsilon).
$$
Then
$$
 d_{\Gamma_0}F \bigl( [\Gamma_\cdot]\bigr)=\bigl[ F\circ \Gamma_\cdot
 \bigr]=\big[\bigl(\varphi_\cdot^*F \bigr)(\Gamma_0)\big],
$$
where the superscript $*$ denotes the pull-back and the square
brackets are used to indicate the equivalence class of curves
determined by the curve they contain. Proceeding in this way, it is
natural to identify the tangent vector $[\,\Gamma_\cdot]$ with the
vector field
$$
 \left.\frac{d}{ds}\right|_{s=0}\varphi_s=\nu_\varphi 
$$
associated to the flow $\varphi$.
\begin{rem}
Only the values of $\nu_\varphi$ on $\Gamma_0$ actually matter but it
is convenient to think of the vector field being defined in at least a
neighborhood of $\Gamma_0$ and sometimes everywhere. Observe that
different vector fields can represent the same tangent vector, but, if
two fields are in the same equivalence class, then they differ by a field
tangential to $\Gamma_0$.
\end{rem}
The following notation will be used from now on
$$
 \langle d_{\Gamma_0}F,\nu_\varphi \rangle = [\varphi^* F],
$$
for the tangential of the section $F$ at $\Gamma_0$.
\subsection{Examples}\label{examples}
{\bf (a)} As a first example, consider $F$ to be a smooth section of the Banach
space ``bundle''\footnote{The term is used in a somewhat loose way
  here in order to appeal to intuition. A formal justification would
  require additional work that is not necessary for the purposes of
  this paper.} 
$$
 E=\coprod_{\Gamma\in\mathcal{U}^{2+\alpha}}buc^{2+\alpha}(\Gamma),
$$
where $F$ is smooth at $\Gamma_0$ if $[s\to F\circ \Gamma_s]$ is
smooth for any smooth curve in $\mathcal{U}^{2+\alpha}$. As a specific
example, take $f\in\operatorname{C}^\infty(\mathbb{R}^n,\mathbb{R}^n)$
and define 
$$
 F(\Gamma)=f\big |_{\Gamma}.
$$
Then $F$ is a smooth section and
$$
 \langle d_{\Gamma_0}F,\nu_\varphi
 \rangle=\left.\frac{d}{ds}\right|_{s=0} f\circ\varphi_s\big
 |_{\Gamma_0}=\nabla f\cdot\nu_\varphi\big |_{\Gamma_0}=
 \partial_{\nu_\varphi} f\big |_{\Gamma_0}.
$$
\begin{rem}
Since any $\Gamma$ in the neighborhood $\mathcal{U}^{2+\alpha}$ of
$\Gamma_0$ is diffeomorphic to it, i.e., there is $\varphi _\Gamma \in
\operatorname{Diff}^{2+\alpha}(\Gamma,\Gamma_0)$, such diffeomorphisms
yield a local trivialization of $E$ via
$$
 buc^{2+\alpha}(\Gamma_0)\times \mathcal{U}\to E,\: (g,\Gamma)\mapsto
 g\circ\varphi_\Gamma.
$$
Now, any diffeomorphism $\varphi _\Gamma$ can be viewed as the restriction of a
general flow $\varphi$, thus providing additional justification for
the approach via flows described above. 
\end{rem}
{\bf (b)} Let $f\in\operatorname{C}^\infty(\mathbb{R}^n,\mathbb{R}^n)$
and consider
\begin{equation}\label{bvpo}
\begin{cases}
  -\Delta u=f&\text{in }\Omega,\\
  u=0&\text{on }\Gamma=\partial \Omega,
\end{cases}
\end{equation}
and let
$$
 F:\mathcal{U}^{2+\alpha}\to\coprod_{\Gamma\in\mathcal{U}^{2+\alpha}}
 bvp^2(\Gamma),\: \Gamma\mapsto (-\Delta _\Omega, \gamma_\Gamma ,f)
$$
be the section of the ``second order boundary value problems bundle'' over
$\mathcal{U}^{2+\alpha}$ corresponding to the above boundary value
problem. Then
$$
 \langle d_{\Gamma_0}F,\nu_\varphi \rangle
 =\left.\frac{d}{ds}\right|_{s=0} \Bigl( \varphi
 _s^*(-\Delta_{\Omega_s})\varphi^s_*,\: \varphi
 _s^*\gamma_{\Gamma_s} \varphi^s_*,\:\varphi^*_sf\Bigr),
$$
where $\varphi^s_*=(\varphi_s^*)^{-1}=(\varphi_s^{-1})^*$.
\begin{rem}
Notice that, in this case, it is assumed that $\varphi _s$ be defined
everywhere so as to be able to transform the operator $-\Delta$
defined on $\Omega$. Observe also that there is nothing geometric (that
is, no identification/trivialization is necessary) in pulling the
problem back to the domain $\Omega_0$ since the original problem on $\Omega_s$ is
equivalent to
\begin{equation}\label{bvpt}
\begin{cases}
  -\varphi_s^*\circ\Delta\circ\varphi^s_* (v)=\varphi^*_s f&\text{in }\Omega_0,\\
  v=0&\text{on }\Gamma_0,
\end{cases}
\end{equation}
for $v=\varphi ^*_su$. This also provides justification for the use of the pull-back
triviliazation introduced earlier as it perfectly matches the
definition of tangential by means of pull-backs. To be more explicit:
if one is interested, as is the case here, in computing $\langle
d_{\Gamma_0}u,\nu_\varphi \rangle,$ 
then one needs to consider $\left.\frac{d}{ds}\right|_{s=0}
\varphi_s^*u$, which necessarily involves determining the solution $v=\varphi^*_su$
of \eqref{bvpt}.
\end{rem}
\begin{rem}\label{independenceoninterior}
Observe that, since only $\nu_\varphi\big |_{\Gamma_0}$ matters, there arises
great freedom in the choice of an extension of the vector field. This
freedom leads to the intuition that, choosing the ``trivial extension'',
the interior of the problem should not have an influence on the domain
variation other than through $u_0$, the solution in $\Omega_0$. More
on this aspect later.
\end{rem}
Returning to the example, one can describe how the solution $u$ depends on
$\Gamma$ or, with moving boundary problems in mind, how
$\partial_{\nu_\Gamma}u$ depends on it. Fix a flow $\varphi_s$ and
solve \eqref{bvpo} on $\Omega_s$ or \eqref{bvpt} to obtain $u=u(s)$ or
$v=v(s)=\varphi_s^*u(s)$, respectively. Then
$$
 \left.\frac{d}{ds}\right|_{s=0}
 \varphi_s^*u=\left.\frac{d}{ds}\right|_{s=0} v(s),
$$
as remarked above. Define
$\mathcal{A}(s)=-\varphi_s^* \Delta \varphi^s_*$ so that
$$
 -\Delta \bigl( v\circ \varphi_s^{-1}\bigr)(x)=\bigl( \mathcal{A}(s)v
 \bigr) \bigl( \varphi_s^{-1}(x)\bigr),\: x\in \Omega_s,
$$
for $v:\Omega_0\to \mathbb{R}$, where
$$
 \Omega_0\ni y=y_s(x)=\varphi _s^{-1}(x).
$$
Using this notation it easily follows that
$$
\mathcal{A}(s)=-\sum_{k,l=1}^n\, \underset{a_{kl}=}{\underbrace{
 \bigl( \sum _{j=1}^n \frac{dy^l_s}{dx^j}
   \frac{dy^k_s}{dx^j} \bigr)}}\, \frac{\partial^2}{\partial y^k \partial y^l}-
 \sum_{l=1}^n\, \underset{b_{l}=}{\underbrace{\bigl( \sum_{j=1}^n
     \frac{\partial^2y^l_s}{(\partial 
   x^j)^2}\bigr)}}\,\frac{\partial }{\partial y^l}.
$$
Differentiating the equations yields
$$
 \left.\frac{d}{ds}\right|_{s=0}
 \mathcal{A}(s)u_0+\mathcal{A}(0)\left.\frac{d}{ds}\right|_{s=0}
 u(s)=\frac{d}{ds} \bigl( f\circ\varphi_s \bigr) 
 =\partial_{\nu_\varphi}f\text{ in }\Omega_0,
$$
and that $\left.\frac{d}{ds}\right|_{s=0} u(s)=0$ on $\Gamma_0$.
Next one needs an expression for $\left.\frac{d}{ds}\right|_{s=0} $ of
the coefficients $\frac{\partial\varphi_s^{-1}}{\partial
  x^j}\circ\varphi_s$ and $\frac{\partial^2\varphi_s^{-1}}{(\partial
  x^j)^2}\circ\varphi_s $.
\begin{lem}
It holds that
$$
 a_{kl}=-\bigl( D\nu_\varphi^\top D\nu_\varphi \bigr)_{kl}\text{ and
 }b_l=-\operatorname{tr}\bigl( D^2\nu^l_\varphi\bigr) ,\: k,l=1,\dots,n
$$
for the vector field $\nu_\varphi$ associated to the flow $\varphi$.
\end{lem}
\begin{proof}
It is plain that $\varphi_s\circ\varphi_s^{-1}=
\operatorname{id}$ implies
$D\varphi_s^{-1}\circ\varphi_s=(D\varphi_s)^{-1}$. Then
$$\begin{cases}
 \dot\varphi_s=\nu_\varphi(\varphi_s),&\\
 \varphi_0=\operatorname{id},&
\end{cases}
$$
implies that
$$\begin{cases}
 D\dot\varphi_s=D\nu_\varphi(\varphi_s)D\varphi _s,&\\
 D\varphi_0=\mathbb{1}.&
\end{cases}
$$
It follows that
\begin{equation*}
\left.\frac{d}{ds}\right|_{s=0}
  (D\varphi_s^{-1}\circ\varphi_s)=-(D\varphi_s)^{-1}D\dot\varphi_s(D\varphi_s)^{-1}\big
  |_{s=0}=-D\dot\varphi_0=-D\nu_\varphi.
\end{equation*}
Next, the relation $D\varphi _s^{-1}=(D\varphi _s)^{-1}$ entails that
$$
 D(D\varphi_s^{-1})=-(D\varphi_s)^{-1}D^2\varphi_s(D\varphi_s)^{-1}.
$$
It also holds that
$$\begin{cases}
 D(D\varphi_s)^{\cdot} =(D^2\varphi_s)^{\cdot} =D^2\nu_\varphi\circ\varphi_s
 \bigl( D\varphi_s,D\varphi_s \bigr) +D\nu_\varphi\circ\varphi_sD^2\varphi_s,&\\
 D^2\varphi_0=0,&
\end{cases}
$$
This yields
\begin{multline*}
 \left.\frac{d}{ds}\right|_{s=0} \bigl(
  D(D\varphi_s^{-1})\circ\varphi_s\bigr) =
 -\left.\frac{d}{ds}\right|_{s=0}
  (D\varphi_s)^{-1}D^2\varphi_0(D\varphi_0)^{-1}-
 (D\varphi_0)^{-1}D^2\dot\varphi_0(D\varphi_0)^{-1}+\\-
 \left.\frac{d}{ds}\right|_{s=0}
  (D\varphi_0)^{-1}D^2\varphi_0(D\varphi_s)^{-1} =-D^2\nu_\varphi
  (\mathbb{1},\mathbb{1})=-D^2\nu_\varphi. 
\end{multline*}
The claim easily follows.
\end{proof}
Summarizing one has that
\begin{align*}
 \mathcal{A}(0)&=-\Delta\text{ on }\Omega_0,\\
 \left.\frac{d}{ds}\right|_{s=0} 
 \varphi_s^*f&=\partial_{\nu_\varphi}f=\nabla f\cdot\nu_\varphi,\\
 \left.\frac{d}{ds}\right|_{s=0} \mathcal{A}(0)&=
 \sum_{k,l=1}^n \bigl( D\nu_\varphi ^\top D\nu_\varphi\bigr)_{lk}
 \frac{\partial ^2}{\partial y^k \partial y^l}+\sum_{l=1}^n
 \operatorname{tr}\bigl( D^2\nu^l_\varphi \bigr) 
 \frac{\partial}{\partial y^l}.
\end{align*}
Finally it is arrived at
\begin{thm}\label{bvpvar}
It holds that
\begin{equation}\label{ellvar}
 \left.\frac{d}{ds}\right|_{s=0} \varphi_s^* \bigl( u(s)\bigr) =(-\Delta
 _{\Omega_0},\gamma_{\Gamma_0})^{-1}\Bigl( \nu_\varphi\cdot\nabla
 f-D\nu_\varphi ^\top  
 D\nu_\varphi : D^2u_0-\operatorname{tr}(D^2\nu_\varphi)\cdot\nabla
 u_0 ,0\Bigr),
\end{equation}
that is, the solution of the homogeneous Dirichlet problem with the
given data.
\end{thm}
In the above theorem the notation $A:B$ was used for
$\operatorname{tr}(A^\top B)$ and symmetric matrices $A,B$. 
Next consider $\partial_{\nu_\Gamma}u$. As the normal derivative is a
function on the boundary, it is natural to expect
$\left.\frac{d}{ds}\right|_{s=0} \partial_{\nu_{\Gamma_s}}u$ not to
depend on interior (to the domain $\Omega_0$) information other than 
$u_0$ itself. Using representation \eqref{ellvar}, however, would seem
to indicate that there be dependence on
$\dot\varphi _s\big |_{\Omega_0}$ as well. It is therefore best to
proceed in a slightly different way. Take $u_0$, the solution of
\eqref{bvpo} in $\Omega_0$, and assume, at first, that $\varphi$ flows
into $\Omega_0$, and look for $u(s)=u_0+\bar u(s)$. Then $\bar u$
satisfies
\begin{equation}\label{bvpbar}
  \begin{cases}
   -\Delta \bar u=0&\text{in }\Omega_s,\\
   \bar u=-u_0\big |_{\Gamma_s}&\text{on }\Gamma_s
  \end{cases}
\end{equation}
and, consequently one has that
$$
 \partial_{\nu_{\Gamma_s}}u(s)=\partial_{\nu_{\Gamma_s}}
 u_0+\partial_{\nu_{\Gamma_s}}\bar u(s).
$$
Next observe that $\bar u(0)\equiv 0$ and that
$$
 \partial_{\nu_{\Gamma_s}}\bar u(s)=-DtN_{\Gamma_s}\bigl( u_0\big
 |_{\Gamma_s}\bigr), 
$$
where $DtN_\Gamma$ denotes the standard Dirichlet-to-Neumann operator
of the domain $\Omega$ with boundary $\Gamma$. It can be concluded that
\begin{multline*}
  \left.\frac{d}{ds}\right|_{s=0} \varphi^*_s
  \bigl( \partial_{\nu_{\Gamma_s}}u(s) \bigr) =\bigl(\cancelto{0}{\left.
  \frac{d}{ds}\right|_{s=0}\varphi^*_s\nu_{\Gamma_s}\bigr)\cdot
  \nabla u_0}+\partial_{\nu_0}\left.\frac{d}{ds}\right|_{s=0}\varphi^*_s
  (u_0\big |_{\Gamma_s})+\\
  -\bigl( \left.\frac{d}{ds}\right|_{s=0} \varphi^*_sDtN_{\Gamma_s}
  \varphi^s_*\bigr) \bigl( \cancelto{0}{u_0\big |_{\Gamma_0}}\bigr)
  -DtN_{\Gamma_0}\bigl( \left.\frac{d}{ds}\right|_{s=0}
  \varphi^*_s(u_0\big |_{\Gamma_s})\bigr)\\
  =\partial _{\nu_0}\partial_{\nu_\varphi}u_0-DtN_{\Gamma_0}
  \bigl( \partial _{\nu_\varphi}u_0\bigr).
\end{multline*}
The first term vanishes in view of Lemma \ref{nuvariation} below.
\begin{thm}\label{dunormalvar}
It holds that
\begin{equation}\label{dunormalvareq}
 \langle d_{\Gamma_0}\partial_{\nu_\Gamma}u,\nu_\varphi \rangle 
 =\left.\frac{d}{ds}\right|_{s=0} \varphi^*_s\partial_{\nu_{\Gamma_s}}u(s)=
 \partial _{\nu_0}\partial_{\nu_\varphi}u_0-
 DtN_{\Gamma_0}\bigl( \partial _{\nu_\varphi}u_0\bigr).    
\end{equation}
\end{thm}
\begin{proof}
It remains to show that the claim is valid for a general flow
$\varphi$. First choose an outward flow $\bar\varphi$ and define
$$
 \overline{\Omega}_\varepsilon =\varphi_\varepsilon(\Omega_0),\:
 \varepsilon>0.
$$
Then, given an arbitrary flow $\varphi$, it will hold that
$$
 \varphi_s(\Omega_0)\subset \overline{\Omega}_\varepsilon \text{
   if }s<<1.
$$
Denoting by $\bar u_\varepsilon$ the solution of \eqref{bvpo} in
$\overline{\Omega}_\varepsilon$, look for $u=\bar u_\varepsilon+\bar
w$, so that $\bar w$ is harmonic in $\Omega_s$ and satisfies
$$
 \bar w=-\bar u_\varepsilon\text{ on }\Gamma_s,
$$
at least for $s<<1$. Retracing the steps of the computation preceding
the theorem, it is arrived at
\begin{multline*}
  \left.\frac{d}{ds}\right|_{s=0} \varphi^*_s\partial_{\nu_{\Gamma_s}}u(s)=\left.
  \frac{d}{ds}\right|_{s=0}\varphi^*_s\bigl( \partial_{\nu_{\Gamma_s}} \bar
  u_\varepsilon \bigr)+\\-\bigl( \left.\frac{d}{ds}\right|_{s=0}
  \varphi^*_sDtN_{\Gamma_s}\varphi^s_*\bigr) \bigl( \bar u_\varepsilon \big
  |_{\Gamma_0}\bigr) -DtN_{\Gamma_0}\bigl(
  \left.\frac{d}{ds}\right|_{s=0}\bar
  u_\varepsilon \big |_{\Gamma_s}\bigr).\quad
\end{multline*}
As this last formula is valid for any $\varepsilon>0$, it can be
inferred that, letting $\varepsilon\to 0$, the claim is indeed valid
since $\bar u_0=u_0$ and thus $\bar u_0\big |_{\Gamma_0}\equiv 0$.
\end{proof}
This shows that, if
$\left.\frac{d}{ds}\right|_{s=0} \varphi^*_s(\partial_{\nu_{\Gamma_s}}u)$ is
computed by means of Theorem \ref{ellvar}, then its independence of
$\dot\varphi_s\big |_{\Omega_0}$ is
obfuscated. There indeed would even appear a possible dependence on
$D^2\nu_\varphi$. This can make calculations for moving boundary
problems less transparent and more cumbersome.
\section{Variations in a parametrized context}
For a given smooth flow $\varphi$ one always has that
$$
 \varphi_s(\Gamma_0)\subset T^{\nu^\delta}_{r_0}(\Gamma_0)
$$
for $s<<1$. Then Lemma \ref{rhoparam} implies that
$$
 \varphi_s(\Gamma_0)=\big\{ \varphi^\delta \bigl(
 y,\rho(s,y)\bigr)\,\big |\, y\in\Gamma_0\big\}=
 \varphi^\delta\circ \bigl( \operatorname{id},\rho(s,\cdot)
 \bigr) (\Gamma_0)=:\Gamma_{\rho(s,\cdot)},
$$
for some $\rho(s,\cdot)\in buc^{2+\alpha}(\Gamma_0)$. It follows that,
in calculations, $\varphi _s$ can be replaced by
$$
 \varphi_\rho:=\Bigl[s\mapsto\varphi^\delta \bigl(\cdot,\rho(s,\cdot)\bigr)\Bigr],
$$
which is a family of diffeomorphisms tracing the same curve
of hypersurfaces. Notice that
$$
 d\varphi_\rho=d\varphi^\delta(\cdot,\rho)+\dot\varphi^\delta d\rho
$$
clearly shows that these are, indeed, diffeomorphisms, provided $\rho$ is
small enough in the $\operatorname{C}^1$-topology. These ``flows'' differ 
merely in their (irrelevant) tangential action. It holds that
$$
 \left.\frac{d}{ds}\right|_{s=0} \varphi^\delta \bigl(
 \cdot,\rho(s,\cdot)\bigr) =\dot\varphi^\delta
 \dot\rho(0,\cdot)=:\dot\rho_0\,\nu^\delta
$$
on $\Gamma_0$. If, on occasion, an extension to a flow on
$\mathbb{R}^n$, denoted by $\Phi^\rho$, is needed, one can
choose one of infinitely many extensions. Here, for the sake of
definiteness, it is proceeded as follows:  for $x\in
T^{\nu^\delta}_{r_0}(\Gamma_0)^\mathsf{c}$ simply set
$$
 \Phi^\rho_s(x)=\Phi^\rho(s,x)\equiv x,\: s\in (-\varepsilon,\varepsilon),
$$
while in $T^{\nu^\delta}_{r_0}(\Gamma_0)$, using the coordinates
$x=\bigl( y(x),r(x)\bigr)$ given by Lemma \ref{smoothtub}, define
\begin{equation}\label{diffeoext}
 \Phi^\rho_s(y,r)=\Phi^\rho \bigl( (y,r),s\bigr)=\bigl(
 y,r+\rho(s,y)\eta(r)\bigr),  
\end{equation}
where $\eta$ is a cut-off function of the type used in the proof of
Lemma \ref{smoothtub}. Clearly $\Phi^\rho_\cdot$ is a family
of diffeomorphisms, which is as smooth as $\rho$ is, and
$$
 \Phi^\rho_s\big |_{\Gamma_0}=\varphi^\delta\circ \bigl( 
 \operatorname{id},\rho(s,\cdot)\bigr).
$$
Again this rests on the assumption that $\rho$ is small which makes
the map $[r\mapsto r+\rho(s,y)\eta(r)]$ invertible for fixed $(s,y)$
thanks to its monotonicity.
\begin{rem}
In calculations it is often convenient to replace $\rho(s,\cdot)$ by
$s\dot\rho_0$ in the above definition as one obtains a curve of
hypersurfaces that is, yes, different, but generates the same tangent
vector.  
\end{rem}
The above considerations can be be summarized as follows
\begin{prop}\label{tangentspace}
 It holds that 
$$
 T_{\Gamma_0}\mathcal{M} ^{2+\alpha}\,\hat{=}\, buc^{2+\alpha}(\Gamma_0),
$$
where the superscript over the equal sign indicates an identification,
which, in this case, is via the map
$$
 h\mapsto h\,\dot\varphi^\delta =h\,\nu^\delta,
 buc^{2+\alpha}(\Gamma_0)\to \mathcal{M} ^{2+\alpha}
$$
\end{prop}
\begin{proof}
Notice that
$$
 \left.\frac{d}{ds}\right|_{s=0} \varphi_s\big
 |_{\Gamma_0}=\nu_\varphi \big |_{\Gamma_0},
$$
as well as that
$$
 \left.\frac{d}{ds}\right|_{s=0} \varphi^\delta \circ \bigl( 
 \operatorname{id},\rho(s,\cdot)\bigr)=\nu^\delta\big
 |_{\Gamma_0}\dot\rho(0,\cdot).
$$
Now, since $\varphi_s$ and $\varphi^\delta\circ \bigl(
\operatorname{id}, \rho(s,\cdot)\bigr)$ yield the same curve of
hypersurfaces, the vector fields $\nu_\varphi\big |_{\Gamma_0}$ and 
$\dot\rho _0\nu^\delta \big |_{\Gamma_0}$ represent the same tangent
vector and, since $\varphi$ can be any flow, the whole tangent space
can be generated in this way. Furthermore the fields $\nu_\varphi$ and
$\nu_{\tilde\varphi}$ associated with two smooth flows $\varphi$ and
$\tilde\varphi$ generating two distinct tangent vectors necessarily 
differ in their normal components at some point of $\Gamma_0$. In this
case, their components in direction of the everywhere transversal
field $\nu^\delta$ will be different, too, showing that the map is
injective.
\end{proof}
\subsection{Variations of the normal vector}
Denote the unit outward normal to $\Gamma_\rho$ by $\nu_\rho$ for any
given $\rho\in buc^{2+\alpha}(\Gamma_0)$. The preceding considerations
and examples point to the necessity of computing
$\langle d_{\Gamma_0}\nu_\Gamma,h\nu^\delta \rangle$. According to the
above observations, this can be performed by evaluating 
$\left.\frac{d}{ds}\right|_{s=0}  \varphi^*_s\nu_{sh}$ for $h\in
buc^{2+\alpha}(\Gamma_0)$ 
\begin{lem}\label{nuvariation}
It holds that
$$
 \left.\frac{d}{ds}\right|_{s=0} \varphi^*_{sh}\,\nu_{sh}=h\,\sum_{j=1}^{n-1}\bigl(\langle
 d_y\nu^\delta(y),\tau^0_j \rangle\big |\nu_0\bigr)\,\tau^0_j-\bigl( \nu^\delta|\nu _0
 \bigr) \sum _{j=1}^{n-1}\partial _jh\,\tau^0 _j
$$
that is, a differential operator of order 1 acting on $h$. Recall
that, by construction, $\nu ^0\big |_{\Gamma_0}=\nu_0=\nu _{\Gamma_0}$.
\end{lem}
\begin{proof}
The notation $\tilde\tau^\rho_j$, $j=1,\dots,n-1$ introduced in
\eqref{taurho} is used here for a basis of tangent vectors in $T
\Gamma_0$ and $\tau^\rho_j$,  $j=1,\dots,n-1$ for their normalized
counterparts. It then follows from
$$
 |\nu_{sh}|=1\text{ and } \tau^{sh}_j\cdot\nu_{sh}=0\text{
   for }j=1,\dots,n-1. 
$$
that
\begin{align*}
 \bigl( \frac{d}{ds} \nu_{sh}\bigr)\cdot
  \tau_j^{sh} &=-\bigl( \frac{d}{ds}
  \tau^{sh}_j\bigr)\cdot \nu_{sh}\\
 \bigl( \frac{d}{ds} \nu_{sh}\bigr)\cdot \nu_{sh} &= 0.
\end{align*}
Evaluating in $s=0$ yields
$$
 \left.\frac{d}{ds}\right|_{s=0} \nu _{sh}=-\sum_{j=1}^{n-1}\bigl(
 \left.\frac{d}{ds}\right|_{s=0}\tau^{sh}_j \cdot \nu_0\bigr)\: \tau ^0_j,
$$
where, by design, $\tau_j^0$, $j=1,\dots,n-1$ is a basis of $T
\Gamma_0$. Thus it is enough to compute 
$\left.\frac{d}{ds}\right|_{s=0} \tau ^{sh}_j$ for $j=1,\dots,n-1$ in
order to compute $\left.\frac{d}{ds}\right|_{s=0} \nu_{sh}$. Next
observe that
\begin{equation*}
 \left.\frac{d}{ds}\right|_{s=0}
 \widetilde{\tau}^{sh}_j=\left.\frac{d}{ds}\right|_{s=0} \bigl(
 |\widetilde{\tau}^{sh}_j|\,\tau^{sh}_j\bigr)
 =\frac{1}{|\widetilde{\tau}^0_j|} \bigl( \left.\frac{d}{ds}\right|_{s=0}
 \widetilde{\tau}^{sh}_j\cdot \widetilde{\tau}^0_j \bigr)
 \tau^0_j +|\widetilde{\tau}^0_j|
 \left.\frac{d}{ds}\right|_{s=0} \tau^{sh}_j,
\end{equation*}
where $\widetilde{\tau}^0_j=\tau^0_j$ is a unit vector for
$j=1,\dots,n-1$. Consequently one has that
$$
 \left.\frac{d}{ds}\right|_{s=0}
 \tau^{sh}_j=\left.\frac{d}{ds}\right|_{s=0} \widetilde{\tau}^{sh}_j-
 \bigl( \left.\frac{d}{ds}\right|_{s=0}
 \widetilde{\tau}^{sh}_j\cdot\tau^0_j \bigr) \tau^0_j,\: j=1,\dots,n-1,
$$
which shows that it is, in fact, enough to compute
$\left.\frac{d}{ds}\right|_{s=0} \widetilde{\tau}^{sh}_j$ for
$j=1,\dots,n-1$. Now
\begin{align*}
 \left.\frac{d}{ds}\right|_{s=0} \widetilde{\tau}^{sh}_j \bigl(
 y,sh(y)\bigr) &=\left.\frac{d}{ds}\right|_{s=0} \langle
 d_y\varphi^\delta\bigl(y,sh(y)\bigr),\tau^0_j \rangle+
 \left.\frac{d}{ds}\right|_{s=0}\Big[\nu^\delta\circ\varphi^\delta
 \bigl(y,sh(y)\bigr) s\partial _j h(y)\Big]\\&=\langle d_y
 \nu^\delta(y),\tau^0_j \rangle \, h(y)+
 \nu^\delta(y) \partial_j h(y),\: y\in \Gamma_0,
\end{align*}
for $j=1,\dots,n-1$, and then
\begin{multline*}
 \left.\frac{d}{ds}\right|_{s=0} \tau^{sh}_j \bigl( y,sh(y)\bigr)
 =\big[ \langle d_y\nu^\delta(y),\tau^0_j \rangle -\bigl(\langle 
 d_y\nu^\delta(y),\tau_j^0 \rangle \cdot\tau^0_j\bigr)\tau^j_0\bigr] h(y)\\+
 \bigl[ \nu^\delta(y)-(\nu^\delta(y)\cdot
 \tau_j^0)\tau_j^0\bigr] \partial_jh(y),\quad y\in \Gamma_0.\qquad
\end{multline*}
Finally
$$
 \left.\frac{d}{ds}\right|_{s=0} \tau^{sh}_j\cdot\nu_0=\bigl(\langle
 d_y\nu^\delta(y),\tau^0_j \rangle\cdot\nu_0\bigr)\, h+(\nu^\delta
 \cdot\nu_0)\,\partial_jh 
$$
and thus
$$
 \left.\frac{d}{ds}\right|_{s=0}
 \nu_{sh}=\bigl(\sum_{j=1}^n\langle d_y\nu^\delta(y),\tau^0_j \rangle\cdot
 \nu_0\bigr)\, h\,\tau^0_j-(\nu^\delta\cdot\nu_0)\sum_{j=1}^{n-1}(\partial_jh)\,\tau_j^0
$$
as claimed.
\end{proof}
\begin{rem}
Notice that $\langle d_y\nu^\delta(y),\tau^0_j \rangle\simeq \lambda _j\tau^0_j$ for
$\delta\simeq 0$ since $\nu^\delta\simeq\nu_0$. Recall that
$\lambda_j$ are the principal curvatures of $\Gamma_0$.
\end{rem}
If $\rho(t,\cdot)$ is a time dependent function, then one can compute
the velocity $V$ in normal direction of the corresponding domains
$\Gamma_{\rho(t,\cdot)}$.  This is clearly an important quantity for
moving boundary problems. One has
\begin{align*}
 V(y)&=\frac{d}{dt}\varphi^\delta \bigl( y,\rho(\cdot,y)\bigr) \cdot\nu
  _{\rho(\cdot,y)}\\ &=\bigl[\frac{d}{dr} \varphi^\delta \bigl(
  y,\rho(\cdot,y)\bigr)\cdot \nu_{\rho(\cdot,y)}\bigr]\,\rho _t(\cdot,y)\\
 &=\bigl[(\nu^\delta \circ\varphi^\delta) \bigl(
   y,\rho(\cdot,y)\bigr)\cdot\nu _{\rho(\cdot,y)}\bigr]\,\rho _t(\cdot,y),\:
 \\&\hspace{3cm}\text{for }y\in \Gamma_0\text{, i.e. }\varphi^\delta \bigl(
  y,\rho(\cdot,y)\bigr)\in \Gamma_{\rho},
\end{align*}
or, for short, $V=(\nu^\delta\cdot\nu_\rho)\,\rho_t$. Notice that
$$
 \bigl( \nu^\delta \circ\varphi^\delta
 \cdot\nu _{\rho}\bigr) \bigl( y,\rho(t,y)\bigr)\simeq
 \nu_{\Gamma_0}\cdot\nu_{\Gamma_0}=1,
$$
uniformly in $y\in \Gamma_0$, if $\delta<<1$ and $t\simeq 0$.
\begin{lem}
It holds that
\begin{multline*}
 \left.\frac{d}{ds}\right|_{s=0} \nu^\delta\circ\varphi^\delta
 \circ(\operatorname{id},sh)\cdot\nu_{sh}=\langle
 d\nu^\delta,\nu^\delta \rangle \cdot\nu_0\,
 h\\+\sum_{j=1}^{n-1}(\nu^\delta \cdot \tau^0_j)\bigl(\langle
 d_y\nu^\delta(y),\tau^0_j \rangle\cdot\nu_0\bigr)\, h
 - \sum_{j=1}^{n-1}(\nu^\delta\cdot\nu_0)(\nu^\delta\cdot\tau^0_j)\,\partial_jh
\end{multline*}
\end{lem}
\begin{proof}
A direct computation using Lemma \ref{nuvariation} gives that the
desired variation amounts to
$$
 \langle d\nu^\delta, \left.\frac{d}{ds}\right|_{s=0} \varphi^\delta
 \rangle \cdot\nu_0+\nu^\delta \cdot \left.\frac{d}{ds}\right|_{s=0} 
 \nu_{sh}=\big[ \langle d\nu^\delta,\nu^\delta \rangle
 \cdot\nu_0\bigr]\, h-\sum_{j=1}^{n-1}
 (\nu^\delta\cdot\nu_0)(\nu^\delta\cdot\tau^0_j)\,\partial_jh
$$
as stated.
\end{proof}
\begin{rem}
Notice that, when $\delta<<1$, one has that
$$
 \langle d\nu^\delta ,\nu^\delta \rangle \simeq 0,\:
 \nu^\delta\cdot\tau^0_j\simeq 
 0\text{, and }\nu^\delta\cdot\nu_0\simeq 1,
$$
uniformly on $\Gamma_0$. It should also be pointed out that this
variation vanishes if $\delta$ can be set to zero.
\end{rem}
\subsection{Examples Revisited}
It is of course possible to interpret the variation of the
solution of a boundary value problem as in Example (b) of Section
\ref{examples} in terms of the identification of Proposition \ref{tangentspace}.
\begin{cor}
Given a smooth flow $\varphi$, let $\varphi_{s h}$ be the
corresponding equivalent curve of diffeomorphisms introduced just
before Proposition \ref{tangentspace}. Then it is already known that $\nu _\varphi$ in
\eqref{ellvar} can be replaced by $h\,\nu^\delta$. The additional
terms $D\nu_\varphi$ and $D^2\nu_\varphi$ can be replaced by
$$
 D(h\,\nu^\delta _0)(x)=(h\circ y)
 (x)D\nu^\delta_0(x)+D \bigl( h\circ y\bigr)(x)\nu^\delta_0(x)
$$
and
$$
 D^2(h\, \nu^\delta_0)(x)=(h\circ y)
 (x)D^2\nu^\delta_0(x)+2D\nu^\delta_0(x)D(h\circ y)(x)+D^2(h\circ
 y)(x)\nu^\delta_0(x),
$$
respectively. Notice that, since $h:\Gamma_0\to \mathbb{R}$ depends on
$y$ only, all of its non vanishing derivatives are tangential ones.
\end{cor}
\begin{rem}
The corollary shows how convenient it is to think in terms of flows or
curves of diffeomorphisms: calculations can be performed in
$\mathbb{R}^n$ and not on the surface. Eventually one can replace the 
generic flow with a parametrized one by means of Proposition
\ref{tangentspace} and the coordinates of Lemma \ref{smoothtub} to
obtain concrete expressions in terms of the parameter function $\rho$.
Recall that $h=\dot\rho(0,\cdot)$.
\end{rem}
\begin{rem}
It should be pointed out that, when the surface $\Gamma_0$ is smooth,
then $\delta$ can be chosen to vanish (no regularization needed). In
that case $\nu^0\big |_{\Gamma_0}=\nu_0$, and consequently, the terms 
$D\nu_0$ and $D^2\nu_0$ have geometric interpretations. E.g. $D\nu_0$
contains information about the curvatures of $\Gamma_0$ and its
Christoffel symbols. 
\end{rem}
\section{Moving Boundary Problems}
Two well-known classical moving boundary problems are the Stefan and
the Hele-Shaw problems. They are used in this section as prototypical
examples to illustrate the benefits of the linearization approach
described the preceding sections which include conciseness and
transparency.
\subsection{Hele-Shaw type problem}
Consider the system
\begin{equation}\label{hse}
  \begin{cases}
    -\Delta u=f&\text{in }\Omega(t)\text{ for }t>0,\\
    u=0&\text{on }\Gamma(t) \text{ for }t>0,\\
    V=-\partial _\nu u&\text{on }\Gamma(t) \text{ for }t>0,\\
    \Gamma(0)=\Gamma_0,&
  \end{cases}
\end{equation}
for $\Gamma_0\in buc^{2+\alpha}$. Then one has the following
\begin{prop}\label{hseprop}
The linearization of \eqref{hse} in $(u(t),\Gamma(t))\equiv
(u_0,\Gamma_0)$, where clearly $u_0$ is the solution of the Poisson
equation with homogeneous Dirichlet condition in $\Omega_0$, is given
by
\begin{equation}\label{linhse}
  \begin{cases}
    -\Delta \bar w=0&\text{in }\Omega_0,\text{ for }t>0\\
    \bar w=-\partial _{\nu_\varphi}u_0&\text{on }\Gamma_0 \text{ for }t>0,\\
    \dot\nu_\varphi\cdot\nu_0=-\partial_{\nu_0}\partial_{\nu_\varphi}u_0
    -\partial _{\nu_0}\bar
    w&\text{on }\Gamma_0 \text{ for }t>0,\\ \nu_\varphi(0)=0,&
  \end{cases}
\end{equation}
where $\nu_\varphi$ denotes the time dependent variation vector field
used to infinitesimaly deform $\Gamma_0$ and $\dot\nu_\varphi$ its
time derivative (see proof below for more detail). In particular, if
$$
 \nu_\varphi=h\,\nu^\delta \text{, then } \dot\nu_\varphi=\dot h\,\nu^\delta,
$$
and \eqref{linhse} reduces to
\begin{equation*}
  \begin{cases}
    -\Delta \bar w=0&\text{in }\Omega_0 \text{ for }t>0,\\
    \bar w=-(\partial _{\nu^\delta }u_0)h&\text{on }\Gamma_0 \text{ for }t>0,\\
    (\nu^\delta \cdot\nu_0)\dot
    h=-(\partial_{\nu_0}\partial_{\nu^\delta}u_0)h-
    \partial _{\nu_0}\bar w&\text{on }\Gamma_0 \text{ for }t>0,\\ h(0,\cdot)=0&
    \text{on }\Gamma_0,
  \end{cases}
\end{equation*}
\end{prop}
\begin{rem}
If $f\geq 0$, the strong maximum principle implies that
$$
 \partial_{\nu_0}u_0>0,
$$
and consequently the same inequality holds for
$\partial_{\nu^\delta}u_0$ since $\delta$ can
be chosen arbitrarily small. This can be used to show that the
operator
$$
 h\mapsto DtN_{\Gamma_0}\bigl( 
 (\partial _{\nu^\delta}u_0)h \bigr),\:
 buc^{2+\alpha}(\Gamma_0)\to buc^{1+\alpha}(\Gamma_0)
$$
generates an analytic semigroup as required by maximal regularity
theory to obtain a solution of the corresponding nonlinear problem. In
this case the linearized system reduces to the single equation
$$
 (\nu^\delta\cdot\nu_0)\,\dot h=DtN_{\Gamma_0}\bigl(
 (\partial_{\nu^\delta
 }u_0)h\bigr)-(\partial_{\nu_0}\partial_{\nu^\delta}
 u_0)h.  
$$
\end{rem}
\begin{rem}
Whenever $f\in\operatorname{C}^\infty(\mathbb{R}^n,
\mathbb{R}^n)$, one has that
$$
 \partial_{\nu^\delta}u_0\in buc^{2+\alpha}.
$$
This regularity is needed to ensure that  multiplication with this
normal derivative of $u_0$ is a continuous operation on
$buc^{2+\alpha}(\Gamma_0)$ and then obtain the generation result of
the previous remark.
\end{rem}
\begin{proof}
Observe that
\begin{align*}
 \Delta\bigl( \partial_{\nu^\delta}u_0)&=\Delta \bigl( 
 \nu^\delta\cdot \nabla u_0\bigr)=\sum_{j=1}^n \partial _j^2 
 \bigl( \nu^\delta_k \partial _ku_0 \bigr)\\
 &=\sum_{j=1}^n\Bigl( \partial_j^2\nu^\delta_k \partial
   _ku_0+2 \partial_j\nu^\delta_k \partial
   _j \partial_ku_0+
   \nu^\delta_k \partial_k\,\underset{=-f}{\underbrace{\Delta u_0}}\Bigr)
   \in buc^\alpha (\Omega_0),
\end{align*}
and that
\begin{align*}
 \partial_{\nu_0} \partial_{\nu^\delta}u_0&=\nu_0^k \partial_k\Bigl(\bigl[
  (\nu^\delta
  \cdot\nu_0)\nu_0+\sum_{j=1}^{n-1}(\nu^\delta\cdot\tau_j^0)\tau_j^0
  \bigr]\cdot\nabla u_0 \Bigr)\\
 &=\partial _{\nu_0} (\nu^\delta\cdot\nu_0) \partial
   _{\nu_0} u_0+(\nu^\delta \cdot\nu_0)\partial_{\nu_0\nu_0}
   u_0\\
 &=(\partial_{\nu_0}\nu^\delta\cdot\nu_0)\partial
   _{\nu_0} u_0-(\nu^\delta\cdot\nu_0)f\in buc^{1+\alpha}(\Gamma_0),
\end{align*}
since
$$
 \partial_{\nu_0} \nu_0=0\text{ and } \partial_{\nu_0\nu_0}u_0+
 \cancelto{0}{\sum_{j=1}^{n-1}\partial_{\tau^0_j\tau^0_j}u_0}=-f
$$
\end{proof}
\begin{proof}(of Proposition \ref{hseprop})
Take a two parameter family of diffeomorphisms $\varphi_{s,t}$ such that
$$
 \varphi_{0,t}\big |_{\Gamma_0}\equiv \operatorname{id}_{\Gamma_0}
$$
and set
$$
 \nu_\varphi=\left.\frac{d}{ds}\right|_{s=0} \varphi_{s,t}\big
 |_{\Gamma_0}\text{ as well as
 }\dot\nu_\varphi=\frac{d}{dt}\left.\frac{d}{ds}\right|_{s=0} 
 \varphi_{s,t}\big |_{\Gamma_0}. 
$$
As follows from the proof of Proposition \ref{tangentspace}, it is
possible to assume without loss of generality that the diffeomorphisms
``flow'' into $\Omega_0$. Then rewrite \eqref{hse} as
\begin{equation*}
  \begin{cases}
    -\Delta \bar u=0&\text{in }\Omega_{s,t}\text{ for }t>0,\\
    \bar u=-u_0 \big |_{\Gamma_{s,t}}&\text{on }\Gamma_{s,t}\text{ for }t>0,\\
    V=-\partial_{\nu_{\Gamma_{s,t}}}(u_0+\bar u)&\text{on
    }\Gamma_{s,t}\text{ for }t>0,\\ \Gamma(0)=\Gamma_0,&
  \end{cases}
\end{equation*}
for $u=u_0+\bar u$, where, again, $u_0$ is the solution of Poisson
equation on $\Omega_0$ with homogeneous Dirichlet condition on the
boundary and
$$
 \Omega_{s,t}=\varphi_{s,t}(\Omega_0)\text{ and
 }\Gamma_{s,t}=\varphi_{s,t}(\Gamma_0).
$$
Then
$$\begin{cases}
 -\mathcal{A}(s)\bar v=-\varphi_{s,t}^*\Delta \varphi^{s,t}_*\bar
 v=0&\text{in }\Omega_0 \text{ for }t>0,\\
 \bar v=-\varphi_{s,t}^* \bigl( u_0\big |_{\Gamma_{s,t}}\bigr) &
 \text{on }\Gamma_0 \text{ for }t>0,
\end{cases}$$
for $\bar v=\varphi_{s,t}^*\bar u$ and
$$
 V=\frac{d}{dt}\varphi_{s,t}\cdot\varphi^*_{s,t}\nu_{\Gamma_{s,t}}=
 -\varphi^*_{s,t}\Bigl[ \partial _{\nu_{\Gamma_{s,t}}}u_0\big
 |_{\Gamma_{s,t}}+ \partial_{\nu_{\Gamma_{s,t}}}\bar u\Bigr]\text{ on
 }\Gamma_0 \text{ for }t>0. 
$$
Taking a variation in $s$ and evaluating in $s=0$ yields
$$\begin{cases}
 -\mathcal{A}(0)\bar w-\left.\frac{d}{ds}\right|_{s=0} 
   \mathcal{A} \cancelto{0}{\bar v(0)}=0&\text{in }\Omega_0 \text{ for }t>0,\\
   \bar w=-\partial _{\nu_\varphi}u_0&\text{on }\Gamma_0 \text{ for }t>0,
\end{cases}$$
and
$$
 \dot\nu_\varphi\cdot\nu_0+0\cdot \left.\frac{d}{ds}\right|_{s=0}
 \varphi^*_{s,t}\nu_{\Gamma_{s,t}}=-\partial
 _{\nu_0}\partial_{\nu_\varphi}u_0 -\partial _{\nu_\varphi}\bar
 w\text{ on }\Gamma_0 \text{ for }t>0,
$$
for $\bar w=\left.\frac{d}{ds}\right|_{s=0} \bar v$ since
$\frac{d}{dt}\varphi\big |_{s=0}\equiv 0$. This system reduces to the claimed one at
the end of the proposition if
$$
 \nu_\varphi=h(t)\, \nu^\delta\text{ and }\dot\nu_\varphi=\dot h(t)\,
 \nu^\delta.
$$
Just use Lemma \ref{smoothtub} to replace the generic curve of
diffeomorphisms with the equivalent $\Phi^\rho_{s,t}$, introduced in
\eqref{diffeoext} based on
$$
 \Phi^\rho_{s,t}\big|_{\Gamma_0}=\varphi^\delta \circ \bigl(
 \operatorname{id},\rho(s,t,\cdot)\bigr),
$$
satisfying
$$
  \Omega_{s,t}=\Phi^\rho_{s,t}(\Omega_0)\text{ and }\Gamma_{s,t}=
  \Phi^\rho_{s,t}(\Gamma_0),
$$
and such that $\left.\frac{d}{ds}\right|_{s=0} \rho(0,t,\cdot)\equiv
h(t,\cdot)$ and $\left.\frac{d}{ds}\right|_{s=0}
\dot\rho(0,t,\cdot)\equiv \dot h(t,\cdot)$.
\end{proof}
Using known results for nonlinear evolution equations \cite{Lun95,DaPG79} one
readily obtains classical local well-posedness results for
\eqref{hse}.
\subsection{A Stefan type problem}
Consider the system
\begin{equation}\label{sp}
  \begin{cases}
    u_t-\Delta u=f&\text{in }\Omega(t)\text{ and }t>0,\\
    u=0&\text{on }\Gamma(t) \text{ and }t>0,\\
    u(0,\cdot)=u_0&\text{in }\Omega_0,\\
    V=-\partial _\nu u&\text{on }\Gamma(t) \text{ and }t>0,\\
    \Gamma(0)=\Gamma_0,&
  \end{cases}
\end{equation}
for $\Gamma_0\in buc^{2+\alpha}$ and $u_0\in
buc^{2+\alpha}(\Omega_0)$.
\begin{prop}
The linearization of \eqref{sp} in
$$
 \bigl( u(t,\cdot),\Gamma(t)\bigr)\equiv (u_0,\Gamma_0) 
$$
is given by
\begin{equation*}
  \begin{cases}
    \bar w_t-\Delta \bar w=0&\text{in }\Omega_0\text{ for }t>0,\\
    \bar w =-\partial_{\nu_\varphi}u_0&\text{on }\Gamma_0\text{ for
    }t>0,\\ \bar w(0,\cdot)=0&\text{in }\Omega_0,\\
    \nu_0\cdot\dot\nu_\varphi=-\partial_{\nu_0}\partial
    _{\nu_\varphi}u_0-\partial_{\nu_\varphi} \bar w&\text{on
    }\Gamma_0\text{ for }t>0,\\ \nu_\varphi(0,\cdot)=0&\text{on
    }\Gamma_0. 
  \end{cases}
\end{equation*}
Again, this reduces, in coordinates, to a system for $(\bar w,h)$ via
$$
 \partial_{\nu_\varphi}u_0=(\partial_{\nu^\delta}u_0)\, h\text{ and }
 \nu_0\cdot\dot\nu_\varphi=(\nu^\delta\cdot\nu_0)\,\dot h
$$
for $\nu_\varphi=h\,\nu^\delta$.
\end{prop}
\begin{proof}
Proceeding as in the previous subsection by means of a two parameter
family of diffeomorphisms, the only change in the calculations is
caused by the time derivative of $u$. In this case the initial
boundary value problem for $u$ is equivalent to
\begin{equation*}
  \begin{cases}
    \varphi^*_{s,t}\circ\bigl(\partial_t -\Delta\bigr)\circ
    \varphi^{s,t}_*\,(\bar v)=0&\text{in
    }\Omega_0\text{ for }t>0,\\
    \bar v=-\varphi^*_{s,t}\bigl( u_0\big
    |_{\Gamma_{s,t}}\bigr)&\text{on }\Gamma_0\text{ for }t>0,\\
    \bar v(0,\cdot)=0&\text{in }\Omega_0, 
  \end{cases}
\end{equation*}
where, again,
$$
 \varphi^{s,t}_*v=u_0+\varphi^{s,t}_*\bar v=u_0+\bar u.
$$
Now
$$
 \varphi^*_{s,t}\,\partial_t\,\varphi^{s,t}_*\bar v=\partial_t\bar
 v+\varphi^*_{s,t} \Bigl( \bigl( \varphi^{s,t}_*\nabla\bar v\bigr)
 \frac{d}{dt}\varphi^{-1}_{s,t}\Bigr)=\partial_t\bar v+\nabla\bar v 
 \cdot V(s,t), 
$$
for $V(s,t)=\varphi^*_{s,t}\dot\varphi^{-1}_{s,t}$. Consequently one
has that
\begin{multline*}
 \left.\frac{d}{ds}\right|_{s=0} \bigl( \partial_t\bar v+\nabla\bar
 v\cdot V(s,t)\bigr)=\partial_t \bigl(\left.\frac{d}{ds}\right|_{s=0} \bar
 v\bigr)+ \nabla \bigl(\left.\frac{d}{ds}\right|_{s=0} \bar v\bigr)\cdot
 V(0,t)+\\+\cancelto{0}{\nabla\bar v(0)}\cdot
 \left.\frac{d}{ds}\right|_{s=0} V(s,t)=\partial_t\bar w,
\end{multline*}
for $\bar w=\left.\frac{d}{ds}\right|_{s=0} \bar v$ since
$$
 V(0,t)=\varphi^*_{0,t}\bigl( \frac{d}{dt} \varphi^{-1}_{0,t}\bigr)=0, 
$$
in view of
$\varphi^{-1}_{0,t}=\varphi_{0,t}=\operatorname{id}_{\Gamma_0}$ for
all $t$. The claim then follows using Proposition \ref{linhse}
\end{proof}
\begin{rem}
If the compatibility conditions
$$
 -\Delta u_0=0\text{ in }\Omega_0\text{ and }u_0=0\text{ on }\Gamma_0,
$$
are satisfied, it is again possible to apply optimal regularity
results to the nonlinear  system for $(u,\rho)$ to obtain local in
time well-posedness for \eqref{sp} in the framework of classical
regularity as well as long time existence and stability of stationary
solutions.
\end{rem}
\begin{rem}
In the described approach one can think of a solution, as far as
$\Gamma(t)$ is concerned, as a curve of diffeomorphisms $\varphi_t$. 
In order to deal with the whole system, it is convenient to think of these as
acting on the whole space. A nice benefit of the linearization
procedure advocated here is that it makes it apparent that the final
result does only depend on $\nu_{\varphi_t}\big |_{\Gamma_0}$, as it could be
expected based on geometric intuition.
\end{rem}

\end{document}